\documentclass[a4paper,12pt]{amsart}

\usepackage[margin=2.5cm, headsep=0.75cm, footskip=0.75cm]{geometry}
\usepackage{amsmath, amsthm, amssymb, mathtools}
\usepackage[hang, flushmargin]{footmisc}
\usepackage[english]{isodate}
\usepackage{enumitem}
\usepackage{xcolor}
\usepackage{pb-diagram}
\usepackage[colorlinks=true, linkcolor=blue, linkbordercolor=blue, citecolor=red, citebordercolor=red, urlcolor=indigo, linktocpage=true]{hyperref}
\usepackage[capitalise, nameinlink, noabbrev, nosort]{cleveref}

\definecolor{indigo}{HTML}{492DA5}

\hypersetup{pdfauthor={Becky Armstrong, Lisa Orloff Clark, and Astrid an Huef}, pdftitle={Inclusions of C*-algebras of graded groupoids}}

\allowdisplaybreaks

\makeatletter\g@addto@macro\bfseries{\boldmath}\makeatother

\setlist[enumerate]{font=\normalfont}
\crefname{enumi}{}{}
\creflabelformat{enumi}{#2(#1)#3}
\crefname{enumii}{}{}
\creflabelformat{enumii}{#2(#1)#3}
\crefname{condition}{condition}{conditions}

\numberwithin{equation}{section}
\crefname{equation}{equation}{equations}

\newtheorem{theorem}{Theorem}[section]

\newtheorem{thm}[theorem]{Theorem}
\crefname{thm}{Theorem}{Theorems}

\newtheorem{lem}[theorem]{Lemma}
\crefname{lem}{Lemma}{Lemmas}

\newtheorem{prop}[theorem]{Proposition}
\crefname{prop}{Proposition}{Propositions}

\newtheorem{cor}[theorem]{Corollary}
\crefname{cor}{Corollary}{Corollaries}

\theoremstyle{definition}

\crefname{dfn}{Definition}{Definitions}

\crefname{notation}{Notation}{Notations}

\theoremstyle{remark}

\newtheorem{rmk}[theorem]{Remark}
\crefname{rmk}{Remark}{Remarks}

\newtheorem{qu}[theorem]{Question}
\crefname{qu}{Question}{Questions}

\crefname{example}{Example}{Examples}

\numberwithin{equation}{section}
\numberwithin{theorem}{section}

\newcommand{\Aa}{\mathcal{A}}
\newcommand{\Bb}{\mathcal{B}}
\newcommand{\Ll}{\mathcal{L}}
\newcommand*{\nb}{\nobreakdash}
\newcommand{\Cst}{\mathrm{C}^*}
\newcommand{\ff}{\mathrm{{f}}} % full
\newcommand{\rr}{\mathrm{{r}}} % reduced
\newcommand{\dd}{\mathrm{{d}}}
\newcommand{\supp}{\operatorname{supp}}
\newcommand{\lsp}{\operatorname{span}}
\newcommand{\id}{\operatorname{id}}
\newcommand{\range}{\operatorname{range}}
\newcommand{\lv}{\lVert}
\newcommand{\rv}{\rVert}

\title[Inclusions of $\Cst$\nb-algebras of graded groupoids]{Inclusions of $\Cst$\nb-algebras of graded groupoids}

\author[Armstrong]{Becky Armstrong}
\author[Clark]{Lisa Orloff Clark}
\author[an Huef]{Astrid an Huef}

\address[B.\ Armstrong]{Mathematical Institute, University of M\"unster, Einsteinstr.\ 62, 48149 M\"unster, GERMANY}
\email{\href{mailto:becky.armstrong@uni-muenster.de}{becky.armstrong@uni-muenster.de}}
\address[L.O.\ Clark and A.\ an Huef]{School of Mathematics and Statistics, Victoria University of Wellington, PO Box 600, Wellington 6140, NEW ZEALAND}
\email{\href{mailto:lisa.clark@vuw.ac.nz}{lisa.clark}, \href{mailto:astrid.anhuef@vuw.ac.nz}{astrid.anhuef@vuw.ac.nz}}

\date{\today}


\thanks{Lisa Clark and Astrid an Huef were supported by the Marsden Fund of the Royal Society of New Zealand (grant number 18-VUW-056). Astrid an Huef thanks Iain Raeburn for telling her about \cite[Theorem~6.2]{KQR} and thanks Camila Sehnem for helpful conversations. The authors also thank Alcides Buss for telling them about \cite{BEM} and \cite{BM}, which lead to a strengthening of the results.}

\begin{document}

\begin{abstract}
We consider a locally compact Hausdorff groupoid $G$ which is graded over a discrete group. Then the fibre over the identity is an open and closed subgroupoid $G_e$. We show that both the full and reduced $\Cst$\nb-algebras of this subgroupoid embed isometrically into the full and reduced $\Cst$\nb-algebras of $G$; this extends a theorem of Kaliszewski--Quigg--Raeburn from the \'etale to the non-\'etale setting. As an application we show that the full and reduced $\Cst$\nb-algebras of $G$ are topologically graded in the sense of Exel, and we discuss the full and reduced $\Cst$\nb-algebras of the associated bundles.
\end{abstract}

\maketitle

\section{Introduction}
We consider a locally compact Hausdorff groupoid $G$ which is graded over a discrete group $\Gamma$. The grading is implemented by a continuous homomorphism $c\colon G\to \Gamma$. In particular, the fibre $G_e\coloneqq c^{-1}(\{e\})$ over the identity $e$ is an open and closed subgroupoid of $G$. Understanding $G_e$ and its $\Cst$\nb-algebras can lead to a better understanding of $G$ and its $\Cst$\nb-algebras. For example, Spielberg proved in \cite[Proposition~1.3]{Spielberg-LMS} that for an \'etale groupoid $G$ \and an amenable group $\Gamma$, if $G_e$ is amenable then so is $G$; the converse follows from \cite[Corollary~5.3.22]{A-DR}. Renault and Williams have generalised Spielberg's result to locally compact Hausdorff groupoids that are second-countable \cite{RW-TAMS}.

When $G$ is \'etale, it follows from \cite{CD} and \cite{RRS} that the reduced $\Cst$\nb-algebra $\Cst_\rr(G)$ is graded in the sense of \cite[Definition~19.2]{Exel's book} and that the conditional expectation implementing the grading is faithful. A crucial ingredient in the proof is that inclusion of $C_c(G_e)$, the continuous and compactly supported complex-valued functions on the fibre $G_e$, extends to an injective homomorphism of $\Cst_\rr(G_e)$ into $\Cst_\rr(G)$ (see, for example, \cite[Lemma~3]{RRS}).

That inclusion also extends to an injective homomorphism with respect to the full norms was proven, again only for \'etale groupoids, by Kaliszewski, Quigg, and Raeburn in \cite[Theorem~6.2]{KQR}. The strategy of \cite{KQR} was to build a pre-Hilbert $C_c(G_e)$-bimodule $X_0$ from $C_c(G)$, and to show that left multiplication by $C_c(G)$ extends to give a homomorphism $L$ of $\Cst(G)$ into the $\Cst$\nb-algebra $\Ll(X)$ of adjointable operators on the completion $X$ of $X_0$. Several norm inequalities then show that inclusion is isometric on $C_c(G)$, and hence extends to an isometric homomorphism. A modern take on the clever proof in \cite{KQR} is that even for non-\'etale groupoids, a ``pre-representation'' of $C_c(G)$ as operators on $X_0$ extends to an $I$-norm bounded homomorphism \cite[Corollary~6.2]{BHM}.

Here we generalise \cite[Theorem~6.2]{KQR} and \cite[Lemma~3]{RRS} to non-\'etale groupoids with a Haar system. Our techniques depend on the existence of a two-sided approximate identity $\{e_\mu\}$ for $\Cst(G)$ in $C_c(G)$ with $\lv e_\mu\rv\le 1$. To get such an approximate identity we invoke \cite[Proposition~1.49]{Dana's book}, which assumes that the unit space $G^{(0)}$ of $G$ is paracompact. Therefore we assume that our groupoids are second-countable locally compact Hausdorff: then so are their unit spaces and hence these spaces are paracompact as needed. Our main theorem is the following.

\begin{thm} \label{thm-inclusions}
Let $G$ be a second-countable locally compact Hausdorff groupoid with a Haar system. Let $\Gamma$ be a discrete group and let $c\colon G\to\Gamma$ be a continuous homomorphism. Set $G_e\coloneqq c^{-1}(\{e\})$. Define $i\colon C_c(G_e)\to C_c(G)$ by
\[
i(f)(x) = \begin{cases}f(x)&\text{if $x \in G_e$}\\
0&\text{else}.
\end{cases}
\]
Then $i$ extends to injective homomorphisms $i_\ff\colon \Cst(G_e)\to \Cst(G)$ and $i_\rr\colon \Cst_\rr(G_e)\to \Cst_\rr(G)$.
\end{thm}

We then show in \cref{cor-app} that both $\Cst(G)$ and $\Cst_\rr(G)$ are topologically graded with conditional expectations induced by restriction of functions to $G_e$. Finally, we discuss what we know about the full and reduced $\Cst$\nb-algebras of the associated Fell bundles.

In a preliminary version of this paper, our results in the reduced setting were weaker. However, after posting the paper on the ArXiv, we learned about the more general results for semigroup actions on a groupoid in \cite[Theorem~5.8]{BM} (the full inclusion) and \cite[Lemma~8.1]{BEM} (the reduced inclusion). Although \cite{BEM} and \cite{BM} have standing assumptions that the groupoids are strongly graded (see \cite[Theorem~3.14]{BM}), the proof techniques used in \cite[Theorem~5.8]{BM} and \cite[Lemma~8.1]{BEM} can be applied to more general graded groupoids. In particular, we were able to adapt the proof of \cite[Lemma~8.1]{BEM} to our setting and thereby strengthen our results.

\subsection*{Set-up} \label{sec-set-up}

Throughout, $G$ is a second-countable locally compact Hausdorff groupoid with a left Haar system $\lambda = \{\lambda^u : u \in G^{(0)}\}$, $\Gamma$ is a discrete group, and $c\colon G\to \Gamma$ is a continuous homomorphism (sometimes called a \emph{cocycle}). Since the image of $c$ is a subgroup of $\Gamma$, we may assume without loss of generality that $c$ is surjective. For $\gamma\in\Gamma$, we set $G_\gamma \coloneqq c^{-1}(\{\gamma\})$. We \emph{choose} the Haar system on the subgroupoid $G_e$ of $G$ to be the restriction $\{\lambda^u|_{G_e} : u \in G^{(0)}\}$ of $\lambda$. We write just $\lambda^u$ for $\lambda^u|_{G_e}$.

We refer to Section~1.4 of \cite{Dana's book} for the definitions of the $I$-norm $\lv\cdot\rv_I$, and of the full and reduced $\Cst$\nb-algebras $\Cst(G)$ and $\Cst_\rr(G)$, respectively. We write $q^G\colon\Cst(G)\to \Cst_\rr(G)$ for the quotient map from the full $\Cst$\nb-algebra $\Cst(G)$ to the reduced $\Cst$\nb-algebra $\Cst_\rr(G)$ of $G$. To distinguish the norms we use $\lv\cdot\rv_\rr$ for the norm on $\Cst_\rr(G)$ and just $\lv\cdot\rv$ for the norm on $\Cst(G)$.

Let $a \in C_c(G)$. For each $\gamma\in\Gamma$, we set $a_\gamma\coloneqq a|_{G_\gamma}$. Since $G_\gamma$ is open and closed, $a_\gamma$ is continuous on $G_\gamma$ with compact support. There exists a finite subset $F$ of $\Gamma$ such that $\cup_{\gamma \in F} G_\gamma$ is an open cover of the support $\supp(a)$ of $a$, and we write $a = \sum_{\gamma\in\Gamma} a_\gamma$, remembering that there are only finitely many nonzero summands since $a_\gamma$ is zero unless $\gamma \in F$.

\section{Isometric inclusions}

In this section we prove \cref{thm-inclusions}, using ideas from \cite[Theorem~6.2]{KQR}. In \cite{KQR}, Kaliszewski, Quigg, and Raeburn construct a groupoid equivalence between a subgroupoid of the skew-product groupoid $G\times_c\Gamma$ and $G_e$, invoke the Equivalence Theorem of \cite{MRW} to make $C_c(G)$ into a pre-Hilbert $C_c(G_e)$-module, and show that left multiplication by $C_c(G)$ acts by bounded adjointable operators on the completion $X$ of this module. Our first observation, in \cref{prop-right-Hilbert} below, is that we do not need to invoke the Equivalence Theorem to get $X$: informed by the \'etale result, we write down analogous formulae for an action and an inner product, and verify that they work.

\begin{prop} \label{prop-right-Hilbert}
Let $a,b \in C_c(G)$ and write $a = \sum_{\gamma \in \Gamma}a_\gamma, \ b = \sum_{\gamma \in \Gamma}b_\gamma$ where $a_\gamma\coloneqq a|_{G_\gamma}, \ b_\gamma\coloneqq b|_{G_\gamma}$, and let $g \in C_c(G_e)$. For $\gamma \in \Gamma$, let $i$ be the natural inclusion of elements of $C_c(G_\gamma)$ in $C_c(G)$ by extension by zero. For $\gamma \in \Gamma$, define
$a_\gamma^*$ and $a_\gamma^*b_\gamma$ by
\begin{equation} \label{eq-defn-restricted-inv-conv}
a_\gamma^*(y) = \overline{a_\gamma(y^{-1})}\quad\text{and}\quad a_\gamma^* b_\gamma(x)\coloneqq\int_{G} i(a_\gamma)^*(z) \, i(b_\gamma)(z^{-1}x)\, \dd\lambda^{r(x)}(z),
\end{equation}
where $x \in G_e$ and $y \in G_{\gamma^{-1}}$.
Then $a_\gamma^* \in C_c(G_\gamma)$ and $a_\gamma^* b_\gamma \in C_c(G_e)$. The formulae
\begin{gather}
(a\cdot g)(x) = \int_{G_e} a(xn) \, g(n^{-1})\, \dd\lambda^{s(x)}(n) \quad \text{ for } x \in G, \ and \label{eq-action}\\
\langle a\,,\, b\rangle = \sum_{\gamma\in\Gamma} a_\gamma^*b_\gamma \label{eq-innerproduct},
\end{gather}
give $C_c(G)$ the structure of an inner-product $C_c(G_e)$-module with $\langle a\,,\, a\rangle\ge 0$ in both $\Cst_\rr(G_e)$ and $\Cst(G_e)$.
\end{prop}

\begin{proof}
For $\gamma \in \Gamma$, the support of the involution $i(a_\gamma)^*$ and the convolution $i(a_\gamma)^* * i(b_\gamma)$ are in $G_{\gamma^{-1}}$ and $G_e$, respectively, and \labelcref{eq-defn-restricted-inv-conv} are the formulae for their restrictions to $G_{\gamma^{-1}}$ and $G_e$. Similarly, $i(g)$ has support in $G_e$ and the formula \labelcref{eq-action} is the formula for the convolution in $C_c(G)$ of $a$ and $i(g)$.

Let $a \in C_c(G)$ and $g, h \in C_c(G_e)$. Since $C_c(G)$ is a $*$-algebra with respect to convolution and involution by \cite[Proposition~1.34]{Dana's book}, it follows, first, that $a\cdot g = a*i(g)$ is again in $C_c(G)$ and, second, that
\[
a\cdot (gh) = a*i(g*h) = a*\big(i(g)*i(h)\big) = \big(a*i(g)\big)*i(h) = (a\cdot g)\cdot h.
\]
Thus \labelcref{eq-action} gives an action of $C_c(G_e)$ on $C_c(G)$.

Next we check that \labelcref{eq-innerproduct} is an inner product. Linearity in the second variable follows from the linearity of convolution in $C_c(G)$, the linearity of $i$, and the linearity of restriction of functions to $C_c(G_e)$. Since $g \in C_c(G_e)$ we have
\[
\langle a\,,\, b\cdot g\rangle = \sum_{\gamma\in\Gamma} a_\gamma^*(b_\gamma\cdot g).
\]
Thus, to see that $\langle a\,,\, b\cdot g\rangle = \langle a\,,\, b\rangle g$, it suffices to show that $a^*(b\cdot g) = (a^*b)g$ when $a,b \in C_c(G)$ with support in $G_\gamma$, and, again, this follows from the associativity of convolution in $C_c(G)$.

We have $\langle a\,,\, b\rangle^* = \left(\sum_{\gamma \in \Gamma} a_\gamma^*b_\gamma \right)^{\!*} = \sum_{\gamma \in \Gamma} b_\gamma^*a_\gamma = \langle b\,,\, a\rangle$. Since a sum of positive elements is positive, we get that $\langle a\,,\, a\rangle = \sum_{\gamma \in \Gamma} a_\gamma^*a_\gamma$ is positive in both $\Cst_\rr(G)$ and $\Cst(G)$. Moreover, if a sum of positive elements is zero, then the summands must be zero, and so $\langle a\,,\, a\rangle = 0$ implies that $a = 0$ in both $\Cst_\rr(G)$ and $\Cst(G)$. Thus \labelcref{eq-innerproduct} is an inner product.
\end{proof}

We will indicate in which $\Cst$\nb-algebra the inner product $\langle \cdot\,,\, \cdot\rangle$ of \labelcref{eq-innerproduct} is taking values using subscripts: $\langle \cdot\,,\, \cdot\rangle_{\Cst(G_e)}$ and $\langle \cdot\,,\, \cdot\rangle_{\Cst_\rr(G_e)}$. Of course, if $a,b \in C_c(G)$, then $\langle a\,,\, b\rangle_{\Cst(G_e)} = \langle a\,,\, b\rangle_{\Cst_\rr(G_e)}$.

\begin{cor} \label{cor-right-Hilbert}
There are completions $X$ and $Y$ of $C_c(G)$ such that $X$ is a right-Hilbert $\Cst(G_e)$-module and $Y$ is a right-Hilbert $\Cst_\rr(G_e)$-module, with formulae for the actions and inner products given by \labelcref{eq-action,eq-innerproduct}. Let $q^{G_e}\colon \Cst(G_e)\to \Cst_\rr(G_e)$ be the quotient map and consider the closed submodule
\[
X_{\ker q^{G_e}} \coloneqq\{ b \in X : \langle a\,,\, b\rangle_{\Cst(G_e)} \in \ker q^{G_e}\text{\ for all\ }a \in X\}
\]
of $X$. Then $Y$ is the quotient $X/X_{\ker q^{G_e}}$.
\end{cor}

\begin{proof}
By \cref{prop-right-Hilbert}, $C_c(G)$ is an inner-product $C_c(G_e)$-module. By, for example, \cite[Lemma~2.16]{tfb}, the completion $X$ in the norm $\lv a\rv_{\Cst(G_e)}\coloneqq \lv\langle a\,,\, a\rangle_{\Cst(G_e)}\rv^{1/2}$ induced by the norm on $\Cst(G_e)$ is a Hilbert $\Cst(G_e)$-module. Similarly, $Y$ is obtained by completing $C_c(G)$ using $\lv a\rv_{\Cst_\rr(G_e)}\coloneqq \lv\langle a\,,\, a\rangle_{\Cst_\rr(G_e)}\rv^{1/2}$. That $Y = X/X_{\ker q^{G_e}}$ follows from the Rieffel correspondence \cite[Theorem~3.22]{tfb}.
\end{proof}

The following observation is probably well known.

\begin{lem} \label{lem-op-descends}
Let $J$ be a closed ideal of a $\Cst$\nb-algebra $A$. Let $X$ be a right-Hilbert $A$-module and consider the closed submodule
\[
X_J \coloneqq \{ y \in X : \langle x\,,\, y\rangle_{A} \in J\text{\ for all\ }x \in X \}.
\]
If $T\in\Ll(X)$, then $T(X_J)\subset X_J$.
\end{lem}

\begin{proof}
Fix $y \in X_J$. Then $\langle x\,,\, y\rangle_A \in J$ for all $x \in X$. Thus
\[
\langle T(y)\,,\, T(y)\rangle_A = \langle T^*T(y)\,,\, y\rangle_A \in J.
\]
By \cite[Lemma~3.23]{tfb} we have $X_J = \{x \in X : \langle x\,,\, x\rangle_A \in J\}$. Thus $T(y) \in X_J$.
\end{proof}

Next we show that left multiplication by $C_c(G)$ acts by bounded adjointable operators on $X$. We start with a technical observation.

\begin{lem} \label{lem-ilt}
Suppose that $\{a_n\}\subset C_c(G)$ and that $a_n\to 0$ as $n\to\infty$ in the inductive limit topology on $C_c(G)$. Then $\langle a_n\,,\, a_n\rangle_{\Cst(G_e)} \to 0$ and $\langle a_n\,,\, a_n\rangle_{\Cst_\rr(G_e)} \to 0$ in the inductive limit topology on $C_c(G_e)$ as $n \to \infty$.
\end{lem}

\begin{proof}
Let $K$ be a compact subset of $G$ such that $\supp (a_n)\subset K$ for all $n$. Then there exists a finite subset $F$ of $\Gamma$ such that $K\subset\cup_{\gamma \in F} G_\gamma$. Write $a_n = \sum_{\gamma \in \Gamma} (a_n)_\gamma$, where each $\supp((a_n)_\gamma)\subset G_\gamma$. Indeed, each $(a_n)_\gamma$ has support in the compact set $K_\gamma\coloneqq K\cap G_\gamma$. Notice that $(a_n)_\gamma\to 0$ uniformly, and then so does the convolution product $(a_n)_\gamma^*(a_n)_\gamma$ with support in the compact set $K_\gamma^{-1}K_\gamma$. Thus
\[
\langle a_n\,,\, a_n\rangle_{\Cst(G_e)} = \sum_{\gamma \in \Gamma}(a_n)_\gamma^*(a_n)_\gamma = \sum_{\gamma \in F}(a_n)_\gamma^*(a_n)_\gamma\to 0
\]
uniformly on the compact subset $\cup_{\gamma \in F} K_\gamma^{-1}K_\gamma$ of $G_e$.
\end{proof}

\begin{prop} \label{prop-L}
Let $X$ be the right-Hilbert $\Cst(G_e)$-bimodule of \cref{cor-right-Hilbert}. Then there exists a homomorphism $L\colon \Cst(G)\to \Ll(X)$ such that $L_a(b) = a * b$ for $a \in C_c(G)\subset \Cst(G)$ and $b \in C_c(G)\subset X$.
\end{prop}

\begin{proof}
We will show that $L\colon C_c(G)\to\Ll(C_c(G))$ is a pre-representation of $C_c(G)$ on $X$ as defined in \cite[Definition~5.1]{BHM}. By \cite[Corollary~6.2]{BHM} this gives an $I$-norm bounded representation $L\colon C_c(G)\to \Ll(X)$, which then extends to give a homomorphism $L\colon \Cst(G)\to\Ll(X)$. (We think this is the modern take on the work Kaliszewski, Quigg, and Raeburn do to prove their \cite[Theorem~6.2]{KQR}.)

We check the three items in \cite[Definition~5.1]{BHM}. For the first, we need to fix $b,d \in C_c(G)$ and show that $a\mapsto \langle b\,,\, L_a(d)\rangle_{\Cst(G_e)}$ is continuous in the inductive limit topology. Let $K$ be a compact subset of $G$ and suppose that $a_n,a \in C_c(G)$ have support in $K$ and that $a_n\to a$ uniformly on $K$. Then
\begin{align*}
\lv\langle b\,,\, L_{a_n}(d)\rangle_{\Cst(G_e)}-&\langle b\,,\, L_a(d)\rangle_{\Cst(G_e)}\rv^2
\\
&= \lv\langle b\,,\,(a_n-a)d\rangle_{\Cst(G_e)} \rv^2\\
&= \lv \langle b\,,\,(a_n-a)d\rangle_{\Cst(G_e)}^* \, \langle b\,,\,(a_n-a)d\rangle_{\Cst(G_e)} \rv\\
&\le \lv \langle b\,,\,b\rangle_{\Cst(G_e)} \rv \lv \langle (a_n-a)d\,,\,(a_n-a)d\rangle_{\Cst(G_e)} \rv
\end{align*}
using the Cauchy--Schwarz inequality. Since $(a_n-a)d\to 0$ in the inductive limit topology on $C_c(G)$, \cref{lem-ilt} implies that $\langle (a_n-a)d\,,\,(a_n-a)d\rangle_{\Cst(G_e)}\to 0$ in the inductive limit topology on $C_c(G_e)$. Now $\langle (a_n-a)d\,,\,(a_n-a)d\rangle_{\Cst(G_e)} \to 0$ in $\Cst(G_e)$, and it follows that
\[
\lv\langle b\,,\, L_{a_n}(d)\rangle_{\Cst(G_e)}-\langle b\,,\, L_a(d)\rangle_{\Cst(G_e)}\rv\to 0.
\]
Thus $a\mapsto \langle b\,,\, L_a(d)\rangle_{\Cst(G_e)}$ is continuous in the inductive limit topology.

For the second, we need to show that for $a_1, a_2, b, d \in C_c(G)$ we have
\[
\langle L_{a_1}(b)\,,\, L_{a_2}(d)\rangle_{\Cst(G_e)} = \langle b\,,\, L_{a_1^*a_2}(d)\rangle_{\Cst(G_e)},
\]
which is precisely the calculation on the middle of page~431 in the proof of \cite[Theorem~6.2]{KQR}.

For the third, we need to see that $\{L_a(b) : a,b \in C_c(G)\}$ is dense in $X$. Since $G_e$ has paracompact unit space, there exists an approximate identity $\{ e_\mu\}\subset C_c(G_e)$ for $\Cst(G_e)$ such that each $ e_\mu$ has norm at most $1$ \cite[Proposition~1.49]{Dana's book}. Fix $a \in C_c(G)$ and $\epsilon>0$. Since $C_c(G)$ is dense in $X$, it suffices to show that $\lv a i(e_\mu)-a\rv_{\Cst(G_e)}<\epsilon$ eventually. There exists a finite subset $F$ of $\Gamma$ such that $a_\gamma\coloneqq a|_{G_\gamma}\ne 0$ implies $\gamma \in F$. Since $F$ is finite, there exists $\mu_0$ such that
\[
\mu \ge \mu_0 \implies \lv a_\gamma^*a_\gamma e_\mu-a_\gamma^*a_\gamma\rv<\frac{\epsilon^2}{2\lvert F \rvert} \ \text{ for all } \gamma \in F.
\]
Now, for $\mu \ge \mu_0$,
\begin{align*}
\lv ai( e_\mu) -a\rv_{\Cst(G_e)}^2
&= \Big\lv\sum_{\gamma \in \Gamma} (ai( e_\mu)-a)_\gamma^* \, (ai( e_\mu)-a)_\gamma \Big\rv\\
&= \Big\lv\sum_{\gamma \in F} e_\mu^* a_\gamma^*a_\gamma e_\mu- e_\mu^* a_\gamma^*a_\gamma - a_\gamma^*a_\gamma e_\mu +a_\gamma^*a_\gamma \Big\rv\\
&\le\sum_{\gamma \in F} \Big(\lv e_\mu^*\rv \lv a_\gamma^*a_\gamma e_\mu -a_\gamma^*a_\gamma\rv + \lv a_\gamma^*a_\gamma-a_\gamma^*a_\gamma e_\mu \rv\Big)\\
&<\sum_{\gamma \in F}\frac{\epsilon^2}{\lvert F \rvert} = \epsilon^2.
\end{align*}
Thus $L\colon C_c(G)\to\Ll(C_c(G))$ is a pre-representation of $C_c(G)$ on $X$, as needed, and so it extends to a homomorphism $L\colon \Cst(G)\to\Ll(X)$.
\end{proof}

We can now prove the first part of \cref{thm-inclusions}. We set
\[
B_e\coloneqq\overline{\{a \in C_c(G) : \supp (a)\subset G_e\}}\subset \Cst(G).
\]

\begin{prop} \label{prop-part1}
Let $L\colon \Cst(G)\to\Ll(X)$ be the homomorphism of \cref{prop-L}. Then $L|_{B_e} \colon B_e\to \Ll(X)$ is injective and inclusion $i\colon C_c(G_e)\to C_c(G)$ extends to an isometric isomorphism $i_\ff\colon \Cst(G_e)\to B_e$.
\end{prop}

\begin{proof}
Since $i$ is a homomorphism that is isometric with respect to the $I$-norms, it extends to a norm-decreasing homomorphism $i_\ff\colon \Cst(G_e)\to \Cst(G)$. Fix $f \in C_c(G_e)$. Then we have $\lv i(f)\rv\le \lv f\rv$ for $f \in C_c(G_e)$. Thus we need to show that $\lv i(f)\rv\ge \lv f\rv$ for $f \in C_c(G_e)$. Notice that for $g \in C_c(G_e)$, we have
\[
\lv i(g)\rv_{\Cst(G_e)}^2 = \Big\lv \sum_{\gamma \in \Gamma} i(g)_\gamma^* \, i(g)_\gamma \Big\rv = \big\lv i(g)_e^* \, i(g)_e \big\rv = \lv g^*g \rv = \lv g \rv^2.
\]
Now fix $f \in C_c(G_e) \setminus \{0\}$. Then $i(f)^*/\lv f \rv$ has norm $1$ in $X$, and
\begin{align*}
\lv L_{i(f)}\rv
&= \sup\{\lv i(f)b\rv_{\Cst(G_e)} : b \in X, \, \lv b\rv_{\Cst(G_e)}\le 1\}\\
&\ge \lv i(f)i(f)^*\rv_{\Cst(G_e)}/\lv f \rv = \lv i(ff^*)\rv_{\Cst(G_e)}/\lv f \rv = \lv ff^*\rv / \lv f \rv = \lv f \rv.
\end{align*}
By \cref{prop-L}, $L$ is norm-decreasing, and so $\lv f \rv\le \lv L_{i(f)} \rv \le \lv i(f) \rv \le \lv f \rv$, giving equalities throughout. Thus $i$ extends to an isometric isomorphism as claimed, and $L$ is isometric on $B_e$.
\end{proof}

Fix $u \in G^{(0)}$. Set $Gu\coloneqq\{x \in G : s(x) = u\}$ and let $\lambda_u$ be the push forward of $\lambda^u$ under inversion, so that
\[
\int_G a(x)\, \dd\lambda_u(x) = \int_G a(x^{-1})\, \dd\lambda^u(x)
\]
for $a \in C_c(G)$. Then the regular representation $\pi_u\colon \Cst(G) \to B\big(L^2(Gu), \lambda_u\big)$ is given for $a \in C_c(G)$ and $h \in C_c(Gu)$ by $\pi_u(a)h = a * h$; that is, by the convolution formula \cite[page~18]{Dana's book}. Let $\pi_u^e$ denote the regular representation of $\Cst(G_e)$ on $L^2((G_e)u, \lambda_u)$.
Let $\gamma \in \Gamma$. Since $G_\gamma$ is open and closed, we can view $L^2((G_\gamma)u, \lambda_u)$ as a closed subspace of $L^2(Gu, \lambda_u)$.

We can now prove the second part of \cref{thm-inclusions}. We set
\[
A_e\coloneqq\overline{\{a \in C_c(G) : \supp (a)\subset G_e\}}\subset \Cst_\rr(G).
\]

\begin{prop}\label{prop-part2}
We have that $\ker q^{G_e} = \ker (q^G\circ i_\ff)$, and consequently the inclusion $i\colon C_c(G_e) \to C_c(G)$ extends to an isomorphism $i_\rr\colon \Cst_\rr(G_e) \to A_e$.
\end{prop}

\begin{proof}
The idea of this proof comes from \cite[Lemma~8.1]{BEM}. Fix $u \in G^{(0)}$. Then $Gu = \bigsqcup_{\gamma\in\Gamma} (G_\gamma)u$, and it follows that the map $U\colon L^2(Gu,\lambda_u) \to \bigoplus_{\gamma \in \Gamma} L^2((G_\gamma)u,\lambda_u)$ defined by $U(h) = (h|_{(G_\gamma)u})_{\gamma \in \Gamma}$ is an isomorphism. A calculation shows that for $a \in C_c(G_e)$ we have
\[
U\pi_u(i_\ff(a))U^* = \bigoplus_{\gamma\in\Gamma}\pi_u^\gamma(a),
\]
where $\pi_u^\gamma\colon C_c(G_e)\to B\big(L^2\big((G_\gamma)u,\lambda_u\big)\big)$ is given by
\[
(\pi_u^\gamma(a)h)(x) = \int_G i(a)(y) \, h(y^{-1}x)\, d\lambda^{r(x)}(y)
\]
for $a \in C_c(G_e)$ and $h \in L^2\big((G_\gamma)u,\lambda_u\big)$.

Define $z_{u,e} \coloneqq u$, and for each $\gamma \in \Gamma {\setminus} \{e\}$ with $(G_\gamma)u \ne \varnothing$, choose $z_{u,\gamma} \in (G_\gamma)u$. Then for each $\gamma \in \Gamma$, $x\mapsto x z_{u,\gamma}$ is a homeomorphism of $(G_e)r(z_{u,\gamma})$ onto $(G_\gamma)u$ which induces an isomorphism
\[
V_\gamma\colon L^2\big((G_\gamma)u,\lambda_u\big)\to L^2\big((G_e)r(z_{u,\gamma}),\lambda_{r(z_{u,\gamma})}\big).
\]
Another calculation shows that $V_\gamma \, \pi_u^\gamma(a) \, V_\gamma^* = \pi_{r(z_{u,\gamma})}^e(a)$, and therefore,
\begin{align*}
\lv i_\ff(a) \rv_r &= \sup\!\big\{ \lv \pi_v(i_\ff(a)) \rv : v \in G^{(0)} \big\} \\
&= \sup\!\big\{ \sup\{ \lv \pi_v^\gamma(a) \rv : \gamma \in \Gamma \} : v \in G^{(0)} \big\} \\
&= \sup\!\big\{ \sup\!\big\{ \big\lv \pi_{r(z_{v,\gamma})}^e(a) \big\rv : \gamma \in \Gamma \big\} : v \in G^{(0)} \big\} \\
&= \sup\!\big\{ \lv \pi_v^e(a) \rv : v \in G^{(0)} \big\} \\
&= \lv a \rv_r.
\end{align*}
It follows that $i$ extends to an injective homomorphism of $\Cst_\rr(G_e)$ onto $A_e$, and that $\ker(q^G\circ i_\ff) = \ker q^{G_e}$.
\end{proof}

This completes the proof of \cref{thm-inclusions}.

\section{Application to topologically graded \texorpdfstring{$\Cst$}{C*}-algebras}

We now show that in the situations of \cref{thm-inclusions}, the groupoid $\Cst$\nb-algebras $\Cst(G)$ and $\Cst_\rr(G)$ are topologically graded. We start by recalling the definitions of \emph{grading} and \emph{topological grading}.

Let $A$ be a $\Cst$\nb-algebra. Following \cite[Definition~16.2]{Exel's book}, we say that \emph{$A$ is graded over $\Gamma$} if there exists a collection $\{A_\gamma : \gamma \in \Gamma\}$ of closed subspaces of $A$ such that
\begin{enumerate}[label=(\roman*), ref=\roman*]
\item \label{grading1} for all $\beta,\gamma \in \Gamma$ we have $A_\beta A_\gamma\subset A_{\beta\gamma}$ and $A_\gamma^* = A_{\gamma^{-1}}$;
\item \label{grading2} $\lsp\{A_\gamma : \gamma \in \Gamma\}$ is dense in $A$; and
\item \label{grading3} the subspaces $\{A_\gamma : \gamma \in \Gamma\}$ are linearly independent.
\end{enumerate}

Each $A_\gamma$ is called a \emph{grading subspace}. In \cite[Definition~19.2]{Exel's book}, $A$ is said to be \emph{topologically graded over $\Gamma$} if there exists a collection $\{A_\gamma : \gamma \in \gamma\}$ of closed subspaces of $A$ satisfying \cref{grading1} and \cref{grading2}, and there exists a bounded linear map $P\colon A\to A$ such that $P = \id$ on $A_e$ and $P = 0$ on $A_\gamma$ when $\gamma\ne e$. If $A$ is topologically graded, then \cref{grading3} holds and $P$ is a conditional expectation onto $A_e$ by \cite[Theorem~19.1]{Exel's book}.

We obtain the following corollary of \cref{thm-inclusions}.

\begin{cor} \label{cor-app}
Let $G$ be a second-countable locally compact Hausdorff groupoid with a Haar system. Let $\Gamma$ be a discrete group and let $c\colon G\to\Gamma$ be a continuous homomorphism. For $\gamma\in\Gamma$, set $G_\gamma\coloneqq c^{-1}(\{\gamma\})$,
\[
A_\gamma\coloneqq\overline{\{f \in C_c(G) : \supp f\subset G_\gamma\}}\subset \Cst_\rr(G),
\]
and
\[
B_\gamma\coloneqq \overline{\{f \in C_c(G) :\supp f\subset G_\gamma\}}\subset \Cst(G).
\]
The $\Cst$\nb-algebras $\Cst_\rr(G)$ and $\Cst(G)$ are topologically graded with grading subspaces $\Aa\coloneqq\{A_\gamma : \gamma\in\Gamma\}$ and $\Bb\coloneqq\{B_\gamma : \gamma\in\Gamma\}$, respectively. In both cases, the conditional expectation is induced by restriction of functions.
\end{cor}

We start by showing that restriction of functions gives linear contractions from $C_c(G)$ to $C_c(G_e)$ which are suitably bounded.

\begin{prop} \label{prop-QP}
There are linear contractions
\[
Q_\ff\colon \Cst(G) \to \Cst(G_e) \ \text{ and } \ Q_\rr\colon \Cst_\rr(G) \to \Cst_\rr(G_e)
\]
such that $Q_\ff(a) = a|_{G_e} = Q_\rr(a)$ for $a \in C_c(G)$, and $q^{G_e}\circ Q_\ff = Q_\rr\circ q^G$.
\end{prop}

\begin{proof}
Fix $a \in C_c(G)$. Since $G_e$ is closed, $a|_{G_e}$ is continuous and has compact support. Let $L\colon \Cst(G)\to\Ll(X)$ be the homomorphism of \cref{prop-L}. We will show that
\begin{equation*}
\lv a|_{G_e} \rv \le \lv a \rv_{\Cst(G_e)}\le \lv L_a \rv;
\end{equation*}
then $\lv a|_{G_e} \rv \le \lv a \rv$, and it follows that restriction extends to a bounded linear map $Q_\ff\colon \Cst(G) \to \Cst(G_e)$ of norm at most $1$.

Since $\sum_{\gamma\in\Gamma} a_\gamma^*a_\gamma\ge a_e^*a_e\ge 0$, we have
\[
\lv a \rv_{\Cst(G_e)}^2 = \Big\lv \sum_{\gamma\in\Gamma} a_\gamma^*a_\gamma \Big\rv\ge \lv a_e^*a_e\rv = \lv a|_{G_e}\rv^2,
\]
giving the first inequality. Let $\{e_\mu\}$ be an approximate identity in $C_c(G_e)$ for $\Cst(G_e)$ such that $\lv e_\mu \rv \le 1$ for all $\mu$ \cite[Proposition~1.49]{Dana's book}. Then $\lv i(e_\mu) \rv_{\Cst(G_e)} \le 1$, and hence
\begin{align*}
\lv L_a\rv^2
&\ge \big\lv L_a(i(e_\mu)) \big\rv_{\Cst(G_e)}^2\\
&= \lv a i(e_\mu)\rv_{\Cst(G_e)}^2\\
&= \lv e_\mu^*\langle a\,,\, a\rangle_{\Cst(G_e)}e_\mu\rv\\
&= \Big\lv e_\mu^*\Big( \sum_{\gamma\in\Gamma} a_\gamma^*a_\gamma\Big) e_\mu \Big\rv \\
&\to \Big\lv \sum_{\gamma\in\Gamma} a_\gamma^*a_\gamma \Big\rv
= \lv a\rv^2_{\Cst(G_e)}.
\end{align*}
Thus $\lv a\rv_{\Cst(G_e)}\le \lv L_a\rv$, giving the second inequality. Thus restriction extends to a linear contraction $Q_\ff\colon \Cst(G) \to \Cst(G_e)$, as claimed.

Next, we will show that $\lv a|_{\Cst(G_e)}\rv_\rr\le \lv a\rv_\rr$. Fix $u \in G^{(0)}$. Let $W$ be the orthogonal projection of $L^2(Gu, \lambda_u)$ onto $L^2((G_e)u, \lambda_u)$. A calculation shows that $\pi_u^e(a|_{G_e}) = W\pi_u(a)W$.
Thus
\[
\big\lv \pi_u^e\big(a|_{G_e}\big) \big\rv = \big\lv W \pi_u(a) W\big\rv \le \lv \pi_u(a) \rv,
\]
since $W$ is a projection. Since $u$ was fixed, this gives $\big\lv a|_{G_e} \big\rv_r \le \lv a \rv_r$. Thus restriction extends to a linear contraction $Q_\rr\colon \Cst_\rr(G) \to \Cst_\rr(G_e)$.

Since $q^{G_e}\circ Q_\ff$ and $Q_\rr\circ q^G$ agree on $C_c(G)$, it follows by continuity that $q^{G_e}\circ Q_\ff = Q_\rr\circ q^G$.
\end{proof}

\begin{proof}[Proof of \cref{cor-app}]
Let $\alpha,\beta\in\Gamma$ and let $a, b \in C_c(G)$ with $\supp(a) \subset G_\alpha$ and $\supp(b) \subset G_\beta$. For $x \in G$, if $(a*b)(x)\ne 0$, then there exists $y\in\supp(a)$ such that $y^{-1}x \in \supp(b)$; that is, $x \in \supp(a)\supp(b)\subset G_\alpha G_\beta\subset G_{\alpha\beta}$. It follows that $A_\alpha A_\beta\subset A_{\alpha\beta}$ and $B_\alpha B_\beta\subset B_{\alpha\beta}$. Further, $a^*$ has support in $G_{\alpha^{-1}}$, and it follows that $A_\alpha^* = A_{\alpha^{-1}}$ and $B_\alpha^* = B_{\alpha^{-1}}$.

We write $a = \sum_{\gamma \in \Gamma} a_\gamma$ where $a_\gamma\coloneqq a|_{G_\gamma}$. There is a finite subset $F$ of $\Gamma$ such that $a_\gamma\ne 0$ implies $\gamma \in F$. Then $a = \sum_{\gamma \in \Gamma} a_\gamma = \sum_{\gamma \in F} i_\gamma(a_\gamma)$, where each $i_\gamma$ is the inclusion map from $C_c(G_\gamma)$ to $C_c(G)$, and so each $i_\gamma(a_\gamma)$ is an element of $C_c(G)$ with support in $G_\gamma$. Thus $C_c(G)$ is a subset of both $\lsp \{A_\gamma : \gamma \in \Gamma\}$ and $\lsp \{B_\gamma : \gamma \in \Gamma\}$, which are thus dense in $\Cst_\rr(G)$ and $\Cst(G)$.

\Cref{prop-QP} gives linear contractions $Q_\rr\colon \Cst_\rr(G) \to \Cst_\rr(G_e)$ and $Q_\ff\colon \Cst(G) \to \Cst(G_e)$ such that $Q_\rr(a) = a|_{G_e} = Q_\ff(a)$ for $a \in C_c(G)$. Composing with the isometric homomorphisms $i_\rr\colon \Cst_\rr(G_e)\to \Cst_\rr(G)$ and $i_\ff\colon \Cst(G_e)\to \Cst(G)$ of \cref{thm-inclusions} gives contractions $P_\rr\coloneqq i_\rr\circ Q_\rr\colon\Cst_\rr(G)\to \Cst_\rr(G)$ and $P_\ff\coloneqq i_\ff\circ Q_\ff\colon\Cst(G)\to \Cst(G)$. In particular, $P_\rr = \id$ on $A_e$ and $P_\rr = 0$ on $A_\gamma$ when $\gamma\ne e$; similarly, $P_\ff = \id$ on $B_e$ and $P_\ff = 0$ on $B_\gamma$ when $\gamma\ne e$. It follows that $P_\rr$ and $P_\ff$ are conditional expectations, and that $\Cst_\rr(G)$ and $\Cst(G)$ are topologically graded.
\end{proof}

\begin{rmk}
Let $G$ be an \'etale groupoid and let $\Gamma$ be an amenable group with a unital subsemigroup $S$. Suppose that $c\colon G \to \Gamma$ is an unperforated cocycle. Then \cite[Theorem~4.6]{CF2021} finds conditions on $G$ and $c$ that are equivalent to realising $\Cst_\rr(G)$ naturally as a covariance algebra of a product system over $S$ in the sense of \cite{Sehnem2019}. In that case $\Cst_\rr(G)$ is topologically graded.\end{rmk}

Recall that a conditional expectation $P$ is \emph{faithful} if $P(a^*a) = 0$ implies $a = 0$. If $G$ is \'etale, then restriction of functions to $G^{(0)}$ induces a faithful conditional expectation $E_\rr\colon \Cst_\rr(G) \to \Cst_\rr(G)$ (see, for example, the footnote on page~116 of \cite{RRS}), and it follows that $P_\rr$ is also faithful because
\[
0 = P_\rr(a^*a) \implies 0 = (E_\rr \circ P_\rr)(a^*a) = E_\rr(a^*a) \implies a = 0.
\]
Note that if $\Cst(G) \ne \Cst_\rr(G)$, then the conditional expectation $E\colon \Cst(G) \to \Cst(G)$ induced by restriction of functions to $G^{(0)}$ is not faithful, and neither is the conditional expectation $P_\ff\colon \Cst(G) \to \Cst(G)$. In the non-\'etale setting, we do not know whether or not $P_\rr$ is faithful. We investigate further.

\begin{prop} \label{prop-tilde L}
Let $X$ and $Y$ be the right-Hilbert bimodules of \cref{cor-right-Hilbert}, let $q\colon X\to Y$ be the quotient map, and let $L\colon \Cst(G)\to\Ll(X)$ be the homomorphism of \cref{prop-L}.
\begin{enumerate}
\item There are homomorphisms $\pi\colon \Ll(X)\to \Ll(Y)$ and $\widetilde L\colon\Cst_\rr(G)\to \Ll(Y)$ such that $\pi(T)(q(x)) = q(T(x))$ for $T \in \Ll(X)$ and $x \in X$, and $\pi\circ L = \widetilde{L}\circ q^G$. In particular, $\widetilde{L}_a(b) = a * b$ for all $a, b \in C_c(G)$.
\item The homomorphism $\widetilde L$ is isometric on $A_e$.
\end{enumerate}
\end{prop}

\begin{proof}
By \cref{lem-op-descends}, $\pi(T)$ is well-defined by the formula $\pi(T)(q(x)) = q(T(x))$. It follows from the Rieffel correspondence (in particular, see \cite[bottom line of page~55]{tfb}) that $\pi$ is $*$-preserving, and it is straightforward to check that $\pi$ is a homomorphism. We claim that $\ker q^G\subset \ker (\pi\circ L)$. It then follows that there is a homomorphism $\widetilde L\colon\Cst_\rr(G)\to \Ll(Y)$ such that $\pi\circ L = \widetilde{L}\circ q^G$.

For the claim, let $a \in \ker q^G$. Let $Q_\ff\colon \Cst(G)\to\Cst(G_e)$ and $Q_\rr\colon \Cst_\rr(G)\to\Cst_\rr(G_e)$ be the contractions of \cref{prop-QP}. Let $b \in C_c(G)$ and notice that we have $Q_\ff(b^*b) = \langle b\,,\, b\rangle_{\Cst(G_e)}$. Let $\{a_n\}\subset C_c(G)$ such that $a_n\to a$ in $\Cst(G)$. Then
\[
\langle L_a(b)\,,\, L_a(b)\rangle_{\Cst(G_e)} = \lim_{n\to\infty} \langle L_{a_n}(b)\,,\, L_{a_n}(b)\rangle_{\Cst(G_e)} = \lim_{n\to\infty}Q_\ff((a_nb)^*(a_nb)) = Q_\ff(b^*a^*ab)
\]
because all the operations involved are continuous. Since $b^*a^*ab\in\ker q^G$ we have
\[
(q^{G_e}\circ Q_\ff)(b^*a^*ab) = (Q_\rr\circ q^G)(b^*a^*ab) = 0,
\]
by \cref{prop-QP}. Thus $\langle L_a(b)\,,\, L_a(b)\rangle_{\Cst(G_e)} = Q_\ff(b^*a^*ab)\in\ker q^{G_e}$. Now
\begin{align*}
\lv\pi(L_a)(q(b))\rv^2_{\Cst_\rr(G_e)} &= \lv q(L_a(b))\rv^2_{\Cst_\rr(G_e)}\\
&= \big\lv\langle q(L_a(b))\,,\, q(L_a(b))\rangle_{\Cst_\rr(G_e)}\rv_\rr\\
\intertext{which, using the Rieffel correspondence again, is}
&= \lv q^{G_e}\big(\langle L_a(b)\,,\, L_a(b)\rangle_{\Cst(G_e)}\big)\rv_\rr = 0.
\end{align*}
Thus $\pi(L_a)(b) = 0$ for all $b \in C_c(G)$. Since $C_c(G)$ is dense in $X$ we have $\pi(L_a) = 0$, and $a \in \ker (\pi\circ L)$. This gives $\widetilde L\colon\Cst_\rr(G)\to \Ll(Y)$ such that $\pi\circ L = \widetilde{L}\circ q^G$. In particular, $\widetilde{L}_a(b) = a*b$ for all $a, b \in C_c(G)$.

To see that $\widetilde{L}$ is isometric on $A_e$ it suffices (by \cref{prop-tilde L}) to show that $\ker (\pi\circ L)\cap B_e\subset \ker q^G$. Fix $a \in \ker(\pi\circ L)\cap B_e$. Let $b \in C_c(G)$. Then $0 = \pi(L_a)(q(b)) = q(L_a(b))$; that is, $L_a(b) \in X_{\ker q^{G_e}}$. Recall from \cref{prop-QP} that restriction extends to a contraction $Q_\rr\colon \Cst_\rr(G) \to \Cst_\rr(G_e)$. Since $G$ has paracompact unit space, there exists an approximate identity $\{ e_\mu\}$ for $C_c(G_e)$ in the inductive limit topology by \cite[Proposition~1.49]{Dana's book}. We have
\[
e_\mu^*i_\ff^{-1}(a) e_\mu = Q_\ff\big( i(e_\mu^*)a i(e_\mu)\big) = \big\langle i(e_\mu)\,,\, L_a( i(e_\mu))\big\rangle_{\Cst(G_e)}\in\ker q^{G_e},
\]
since $L_a( i(e_\mu)) \in X_{\ker q^{G_e}}$. By taking limits we get that $i_\ff^{-1}(a) \in \ker q^{G_e}$, and then $a \in i_\ff(\ker q^{G_e})\subset \ker q^G$. Now $\widetilde{L}|_{A_e}$ is isometric.
\end{proof}

\begin{lem} \label{lemma-kernel}
Let $L\colon \Cst(G)\to \Ll(X)$ and $\widetilde{L}\colon \Cst_\rr(G)\to \Ll(Y)$ be the homomorphisms of \cref{prop-L,prop-tilde L}. We have
\[
\ker\widetilde{L} = \{a \in \Cst_\rr(G) : P_\rr(a^*a) = 0\} \quad\text{and}\quad \ker L = \{a \in \Cst(G) : P_\ff(a^*a) = 0\}.
\]
\end{lem}

\begin{proof}
We prove $\ker\widetilde{L} = \{a \in \Cst_\rr(G) : P_\rr(a^*a) = 0\}$; for $\ker L$, just replace $P_\rr$ by $P_\ff$ and $\widetilde{L}$ by $L$ in the argument below.

By \cref{prop-part2}, inclusion extends to a homomorphism $i_\rr\colon \Cst_\rr(G_e) \to \Cst_\rr(G)$. Fix $\alpha \in \Gamma$. We start by showing that for all $b \in C_c(G)$ with support in $G_\alpha$ and $a \in \Cst_\rr(G)$, we have
\begin{equation} \label{eq-Ruy}
i_\rr\big(\langle b\,,\widetilde{L}_a(b)\rangle_{\Cst_\rr(G_e)}\big) = b^*P_\rr(a)b.
\end{equation}
(We got the idea for this from \cite[Proposition~17.12]{Exel's book}.) First let $a,b \in C_c(G)$ with $\supp(b)\subset G_\alpha$. Then
\[
\langle b\,,\widetilde{L}_a(b)\rangle_{\Cst_\rr(G_e)} = \langle b\,,\\ab\rangle_{\Cst_\rr(G_e)} = \sum_{\gamma\in\Gamma}b_\gamma^*(ab)_\gamma = b_\alpha^*(ab)_\alpha = b_\alpha^*a_eb_\alpha
\]
and $i_\rr(b_\alpha^*a_eb_\alpha) = b^*P_\rr(a)b$.
Now consider $a \in \Cst_\rr(G)$ and choose $\{a_n\}\subset C_c(G)$ such that $a_n\to a$. Then
\[
i_\rr\big(\langle b\,,\widetilde{L}_a(b)\rangle_{\Cst_\rr(G_e)}\big) = \lim_{n\to\infty}i_\rr\big(\langle b\,,\widetilde{L}_{a_n}(b)\rangle_{\Cst_\rr(G_e)}\big) = \lim_{n\to\infty}b^*P_\rr(a_n)b = b^*P_\rr(a)b,
\]
giving \labelcref{eq-Ruy}. To see that $\{a \in \Cst_\rr(G) : P_\rr(a^*a) = 0\} = \ker\widetilde{L}$, we observe that
\begin{align*}
b^*P_\rr(a^*a)b = i_\rr\big(\langle b\,,\widetilde{L}_{a^*a}(b)\rangle_{\Cst_\rr(G_e)}\big) = i_\rr\big(\langle \widetilde{L}_{a}(b)\,,\widetilde{L}_{a}(b)\rangle_{\Cst_\rr(G_e)}\big).
\end{align*}
Thus
\begin{align*}
\widetilde{L}_{a} = 0 &\iff \widetilde{L}_{a}(b) = 0 \text{ for all $b \in C_c(G)$}\\
&\iff \widetilde{L}_{a}(b) = 0 \text{ for all $b \in C_c(G)$ with $\supp(b)\subset G_\alpha$ for some $\alpha\in\Gamma$}\\
&\iff 0 = b^*P_\rr(a^*a)b \text{ for all $b \in C_c(G)$ with $\supp(b)\subset G_\alpha$ for some $\alpha\in\Gamma$}\\
&\iff 0 = P_\rr(a^*a),
\end{align*}
by using an approximate identity contained in $C_c(G)$ at the last step.
\end{proof}

Notice that the collections $\Aa$ and $\Bb$ of \cref{cor-app} are Fell bundles, in the sense of \cite[Definitions~16.25 and 17.6]{Exel's book}. If $P_\rr$ is faithful, then by \cite[Proposition~19.8]{Exel's book} the reduced $\Cst$\nb-algebra $\Cst_\rr(G)$ is isomorphic to the reduced $\Cst$\nb-algebra $\Cst_\rr(\Aa)$ of the Fell bundle $\Aa$. However, in general, all we can say is in the following proposition.

\begin{prop} \label{prop-bundles A and B}
Let $L$ and $\widetilde L$ be the homomorphisms of \cref{prop-L,prop-tilde L}, and let $\Aa$ and $\Bb$ be the Fell bundles of \cref{cor-app}. Then $\Cst_\rr(\Aa)$ is isomorphic to the range of $\widetilde L$, $\Cst_\rr(\Bb)$ is isomorphic to the range of $L$, and $\Cst(\Bb)$ is isomorphic to $\Cst(G)$.
\end{prop}

\begin{proof}
The first isomorphism theorem gives an isomorphism of $\Cst(G)/\ker L$ onto the range of $L$. Fix $\gamma \in \Gamma$ and $b \in B_\gamma$. Then $b^*b \in B_e$, and since $L|_{B_e}$ is isometric, we have
\[
\lv L_b \rv^2 = \lv L_b^* L_b \rv = \lv L_{b^*b} \rv = \lv b^* b \rv = \lv b \rv^2.
\]
Hence $L|_{B_\gamma}$ is isometric for all $\gamma \in \Gamma$. It follows that $B_\gamma/\ker L|_{B_\gamma} = B_\gamma$. Thus $\Cst(G)/\ker L$ has grading $\Bb = \{B_\gamma : \gamma \in \Gamma\}$. By \cref{lemma-kernel} we have $\ker L = \{b \in \Cst(G) : P_\ff(b^*b) = 0\}$. Since any element of a $\Cst$\nb-algebra can be written as a sum of positive elements, it follows that there is a conditional expectation $\widetilde P_\ff$ on $\Cst(G)/\ker L$ defined by $\widetilde P_\ff(b+\ker L) \coloneqq P_\ff(b)$. Since $\widetilde P_\ff$ is faithful, $\Cst(G)/\ker L$ (and hence the range of $L$) is isomorphic to $\Cst_\rr(\Bb)$ by \cite[Proposition~19.8]{Exel's book}. Similarly, $\Cst_\rr(\Aa)\cong \range \widetilde{L}$.

Next we verify that $\Cst(G)$ has the universal property of $\Cst(\Bb)$. For each $\gamma\in\Gamma$, let $\id_\gamma\colon B_\gamma\to \Cst(G)$ be the identity map. Then $\{\id_\gamma : \gamma\in\Gamma\}$ is a representation of $\Bb$ in $\Cst(G)$ which generates $\Cst(G)$ in the sense that $\lsp\{\id_\gamma(b) : \gamma \in \Gamma, b \in B_\gamma\}$ is dense in $\Cst(G)$.

Let $\{\pi_\gamma : \gamma \in \Gamma\}$ be a representation of $\Bb$ in a $\Cst$\nb-algebra $D$. We need to show that there is a homomorphism $\pi\colon \Cst(G)\to D$ such that $\pi\circ\id_\gamma = \pi_\gamma$ for all $\gamma\in\Gamma$. For $a \in C_c(G)$, define $\pi\colon C_c(G)\to D$ by $\pi(a) = \sum_{\gamma\in\Gamma}\pi_\gamma(a_\gamma)$. It is straightforward to check that $\pi$ is a $*$-homomorphism. We will argue that $\pi$ is continuous in the inductive limit topology, and hence extends to give a homomorphism $\pi\colon \Cst(G)\to D$. Then since $\pi\circ\id_\gamma$ and $\pi_\gamma$ agree on $C_c(G)$, they agree on $\Cst(G)$.

By \cref{thm-inclusions}, $i_\ff\colon \Cst(G_e)\to B_e$ is an isometric isomorphism, and we can identify $B_e$ with $\Cst(G_e)$. Then $\pi_e\colon \Cst(G_e)\to D$ is a homomorphism on a groupoid $\Cst$\nb-algebra, and hence is $I$-norm bounded on $C_c(G_e)$. We claim that each $\pi_\gamma$ is $I$-norm bounded as well. For $a \in B_\gamma\cap C_c(G)$ we have
\[
\lv\pi_\gamma(a)\rv^2 = \lv\pi_\gamma(a)^*\pi_\gamma(a)\rv = \lv\pi_e(a^*a)\rv \le \lv a^*a\rv_I \le \lv a^*\rv_I\lv a\rv_I = \lv a\rv_I^2
\]
because the $I$-norm is submultiplicative by \cite[page~16]{Dana's book}.

Now suppose that $a, a_n \in C_c(G)$ with support in a compact subset $K$ of $G$ and that $a_n\to a$ as $n\to\infty$. Fix $\epsilon>0$. Write
\[
a = \sum_{\gamma\in\Gamma}a_\gamma \ \text{ and } \ a_n = \sum_{\gamma \in \Gamma} (a_n)_\gamma.
\]
Let $F$ be a finite subset of $\Gamma$ such that if $a_\gamma\ne 0$ or $(a_n)_\gamma\ne 0$, then $\gamma \in F$. For every $\gamma \in F$ we have $(a_n)_\gamma\to a_\gamma$ uniformly on $K$ and hence $\lv(a_n)_\gamma-a_\gamma\rv_I\to 0$. Since $F$ is finite, there exists $N$ such that $n>N$ implies that $\lv(a_n)_\gamma-a_\gamma\rv_I <\epsilon/\lvert F \rvert$ for all $\gamma \in F$. Then for $n>N$ we have
\begin{align*}
\lv\pi(a_n)-\pi(a)\rv &= \Big\lv\pi\Big(\sum_{\gamma \in \Gamma}\big((a_n)_\gamma-a_\gamma\big)\Big)\Big\rv
\le \sum_{\gamma \in F}\big\lv \pi_\gamma\big( (a_n)_\gamma-a_\gamma \big) \big\rv\\
&\le \sum_{\gamma \in F}\big\lv\big( (a_n)_\gamma-a_\gamma \big)\big\rv_I
\le \sum_{\gamma \in F} \epsilon/\lvert F \rvert = \epsilon.
\end{align*}
Thus $\pi$ is continuous in the inductive limit topology and hence extends as claimed. Now $\Cst(G)$ has the universal property of $\Cst(\Bb)$, and hence they are isomorphic.
\end{proof}

\begin{rmk}
That $\Cst_\rr(\Aa)$ and $\Cst_\rr(\Bb)$ are isomorphic to the ranges of $\widetilde L$ and $L$, respectively, is no surprise. For example, by definition, $\Cst_\rr(\Aa)$ is the range of the regular representation $\Lambda\colon \Cst(\Aa)\to \Ll(\ell^2(\Aa))$, where $\ell^2(\Aa)$ is a Hilbert bimodule isomorphic to $Y$ and $\Lambda$ is implemented by left multiplication.
\end{rmk}

\begin{qu}
Suppose that $G$ is \'etale. The conditional expectation $P_\rr$ is faithful on $\Cst_\rr(G)$, and hence $\Cst_\rr(G)\cong\Cst_\rr(\Aa)$ by \cite[Proposition~19.8]{Exel's book}. What happens if $G$ is not \'etale?
\end{qu}

To sum up, we give the following commutative diagram where $\rho$ is implemented by $\{q^G|_{B_\gamma} : \gamma \in \Gamma\}$ and $\pi$ is the homomorphism of \cref{prop-tilde L}. We do not in general know when either of the two bottom horizontal arrows are isomorphisms.
\[
\begin{diagram}
\dgARROWLENGTH=0.8\dgARROWLENGTH
\node{\Cst(\Bb)\cong\Cst(G)}\arrow{e}\arrow{s,r}{\rho}
\node{\range L\cong\Cst_\rr(\Bb)}
\arrow{se,r}{\pi}
\\
\node{\Cst(\Aa)}\arrow{e}
\node{\Cst_\rr(G)}\arrow{e}
\node{\range\widetilde L\cong \Cst_\rr(\Aa).}
\end{diagram}
\]

\vspace{2ex}

\vspace{2ex}

\begin{thebibliography}{00}

\bibitem{A-DR} C.~Anantharaman-Delaroche and J.~Renault, \textit{Amenable Groupoids}, Monographies de L'Enseignement Math\'{e}matique vol.\ 36, L'Enseignement Math\'{e}matique, Geneva, 2000.

\bibitem{BEM} A.~Buss, R.~Exel, and R.~Meyer, Reduced $\Cst$\nb-algebras of Fell bundles over inverse semigroups, \textit{Israel J.~Math.}\ \textbf{220} (2017), 225--274.

\bibitem{BHM} A.~Buss, R.D.~Holkar, and R.~Meyer, A universal property for groupoid $\Cst$\nb-algebras.~I, \textit{Proc.\ London Math.\ Soc.}\ \textbf{117} (2018), 345--375.

\bibitem{BM} A.~Buss and R.~Meyer, Inverse semigroup actions on groupoids, \textit{Rocky Mountain J.~Math.}\ \textbf{47} (2017), 53--159.

\bibitem{CD} L.O.~Clark and E.~Dawson, Strong gradings on Leavitt path algebras, Steinberg algebras and their C*-completions, \textit{J.~Algebr.\ Comb.} (2022), 1--12.

\bibitem{CF2021} L.O.~Clark and J.~Fletcher, Groupoid algebras as covariance algebras, \textit{J.~Operator Theory}\ \textbf{85} (2021), 347--382.

\bibitem{Exel's book} R.~Exel, \textit{Partial Dynamical Systems, Fell bundles and Applications}, Mathematical Surveys and Monographs, vol.\ 224, American Mathematical Society, Providence, RI, 2017.

\bibitem{KQR} S.~Kaliszewski, J.~Quigg, and I.~Raeburn, Skew products and crossed products by coactions, \textit{J.~Operator Theory}\ \textbf{46} (2001), 411--433.

\bibitem{MRW} P.~Muhly, J.~Renault, and D.P.~Williams, Equivalence and isomorphism for groupoid $\Cst$\nb-algebras, \textit{J.~Operator Theory}\ \textbf{17} (1987), 3--22.

\bibitem{tfb} I.~Raeburn and D.P.~Williams, \textit{Morita Equivalence and Continuous-Trace $\Cst$\nb-Algebras}, Mathematical Surveys and Monographs, vol.\ 60, American Mathematical Society, Providence, RI, 1998.

\bibitem{RW-TAMS} J.~Renault and D.P.~Williams, Amenability of groupoids arising from partial semigroup actions and topological higher rank graphs, \textit{Trans.\ Amer.\ Math.\ Soc.}\ \textbf{369} (2017), 2255--2283.

\bibitem{RRS} A.~Rennie, D.~Robertson, and A.~Sims, Groupoid algebras as Cuntz--Pimsner algebras, \textit{Math.\ Scan.}\ \textbf{120} (2017), 115--123.

\bibitem{Sehnem2019} C.F.~Sehnem, On $\Cst$\nb-algebras associated to product systems, \textit{J.~Funct.\ Anal.}\ \textbf{277} (2019), 558--593.

\bibitem{Spielberg-LMS} J.~Spielberg, $\Cst$\nb-algebras for categories of paths associated to the Baumslag--Solitar groups \textit{J.~London Math.\ Soc.}\ \textbf{86} (2012), 728--754.

\bibitem{Dana's book} D.P.~Williams, \textit{A Tool Kit for Groupoid $\Cst$\nb-Algebras}, Mathematical Surveys and Monographs, vol.\ 241, American Mathematical Society, Providence, RI, 2019.

\end{thebibliography}
\end{document}